\documentclass[12pt,a4paper,oneside,reqno,english]{amsart}
\usepackage[T1]{fontenc}
\usepackage[utf8]{inputenc}
\usepackage{verbatim}
\usepackage{amsmath}
\usepackage{amstext}
\usepackage{amsthm}
\usepackage{stmaryrd}
\usepackage{setspace}
\makeatletter

\pdfpageheight\paperheight
\pdfpagewidth\paperwidth

\numberwithin{equation}{section}
\numberwithin{figure}{section}
\theoremstyle{plain}
\newtheorem{thm}{\protect\theoremname}[section]
\theoremstyle{remark}
\newtheorem{rem}[thm]{\protect\remarkname}
\theoremstyle{plain}
\newtheorem{lem}[thm]{\protect\lemmaname}
\theoremstyle{definition}
\newtheorem{definition}{Definition}

\usepackage{amsfonts}
\usepackage{amsthm}
\usepackage{epsfig}
\usepackage{mathrsfs}

\@ifundefined{definecolor}
 {\usepackage{color}}{}

\usepackage{bm}

\usepackage{graphicx}
\usepackage{rotating}
\usepackage{longtable}
\usepackage{pdflscape}
\usepackage{booktabs}
\usepackage{multirow}
\usepackage{float}
\usepackage{makecell}

\usepackage[ruled,linesnumbered]{algorithm2e}

\usepackage[noabbrev,capitalise]{cleveref}

\makeatletter

\newcommand{\Rmnum}[1]{\expandafter\@slowromancap\romannumeral#1@}\makeatother

\numberwithin{equation}{section}

\allowdisplaybreaks

\newcommand{\defs}{:=}

\DeclareSymbolFont{lettersA}{U}{pxmia}{m}{it}
\DeclareMathSymbol{\piup}{\mathord}{lettersA}{"19}

\makeatother

\makeatother

\usepackage{babel}
\providecommand{\lemmaname}{Lemma}
\providecommand{\problemname}{Problem}
\providecommand{\remarkname}{Remark}
\providecommand{\theoremname}{Theorem}

\begin{document}
\title[]{On the singular layer for the 1-D steady Euler system with heat conductivity in a finite nozzle as the heat conduction coefficient vanishes}
\author{Su Jiang}

\address{S. Jiang: School of Mathematics and Statistics,
	Xiamen University of Technology, Xiamen 361024, China}
\email{\texttt{js17854295180@alumni.sjtu.edu.cn}}

\keywords{Singular layer; Heat conductivity; Asymptotic analysis; Euler equations; Flat Nozzle}
\subjclass[2020]{35A21, 35Q31, 80M35, 35F60}

\date{\today}
\newpage
\begin{abstract}
This paper investigates the singular limit problem of the one-dimensional steady compressible Euler system with heat conduction in a finite nozzle. The existence of the heat conduction solutions were established and boundary layers at one or two nozzle endpoints occur as the coefficient of heat conductivity goes to zero. The asymptotic behavior of the solutions vary depending on the state at the entrance and the value of the pressure at the exit, which classify the solutions into purely supersonic solution, purely subsonic solution and transonic solution. Moreover, we perform a formal asymptotic expansion of solutions and derive the boundary layer equations.

\end{abstract}

\maketitle
\tableofcontents{}

\section{Introduction}\label{section1}
In this paper we are concerned with the existence of the solutions for one-dimensional (1-D) steady flows with heat conduction in a finite nozzle as well as their asymptotic behavior as the coefficient of heat conductivity goes to zero. The flow is assumed to be governed by the following equations, called (1-D) steady Euler system with heat conductivity in this paper: 
\begin{align}
	 &\partial_{x}(\rho u)        =0,\label{Euler11} \\
	&\partial_{x}(\rho u^{2}+p)  =0,\label{Euler12} \\
	&\partial_{x}(\rho u \Phi)   = \epsilon \partial_{xx}e,\label{Euler13}
\end{align}
where $u$, $\rho$, $p$ stand for the velocity, density and the pressure of the flow respectively, and $\Phi:=\frac{1}{2}u^{2}+e$ is the Bernoulli constant with $\displaystyle e:=\frac{1}{\gamma-1}\cdot \frac{p}{\rho}$ being the internal energy. Here, $\epsilon$ is the coefficient of the heat conductivity. 
Let $c:=\sqrt{\partial_{\rho}p}=\sqrt{\gamma p/ \rho}$ be the sound speed and $M:=\displaystyle\frac{|u|}{c}$ be the Mach number, then the flow is supersonic in case $u>c$, or equivalently, $M>1$, and it is subsonic in case $u<c$, or equivalently, $M<1$.

Assume
\begin{equation}
    \mathcal{N}:=\{x : 0<x<1\}
\end{equation}
be the bounded nozzle. Set $U=(u, \rho, p)$ be the state of the fluid. Then for given positive constants $q_{in}$, $\rho_{in}$, $p_{in}$, $p_{e}$, we consider a boundary value problem for \eqref{Euler11}, \eqref{Euler12}, \eqref{Euler13} in the nozzle with boundary conditions 
\begin{equation}\label{bd}
    U(0)=U_{in}, \quad p(1)=p_{e},
\end{equation}
where $U_{in}:=(q_{in}, \rho_{in}, p_{in})$. The natural question arises here: does the solution exist for equations \eqref{Euler11}, \eqref{Euler12}, \eqref{Euler13} subject to the general boundary conditions \eqref{bd}? If exists, is the solution unique?
Furthermore, as $\epsilon$ goes to $0$, equations \eqref{Euler11}, \eqref{Euler12}, \eqref{Euler13} formally converge to the following ideal Euler system:
\begin{equation}\label{Euler}
	\left\{
	\begin{aligned}
		 & \partial_{x}(\rho u)        =0, \\
		 & \partial_{x}(\rho u^{2}+p)  =0, \\
		 & \partial_{x}(\rho u \Phi)   =0.
	\end{aligned}
	\right.
\end{equation}
However, the prescribed boundary conditions \eqref{bd} are generally overdetermined for the limiting problem.  Then the limit may not be a Euler solution satisfying  boundary conditions \eqref{bd}. So if the solution exists, what is the asymptotic behavior of the solution as the coefficient of the heat conductivity $\epsilon$ tends to $0$? In fact, layers at the entrance or the exit of the nozzle may develop, the above boundary value problem \eqref{Euler11}, \eqref{Euler12}, \eqref{Euler13}, \eqref{bd} with heat conductivity and the analysis for the asymptotic behavior of the solution become a singular limit problem. 

Our goal in this paper is trying to establish the existence of the solution to system \eqref{Euler11}, \eqref{Euler12}, \eqref{Euler13} with given boundary conditions \eqref{bd} at the entrance and the exit of the nozzle, then to analyze the asymptotic behavior of the solutions as the coefficient of the heat conductivity vanishes. In fact, the existence result of the solutions for system \eqref{Euler11}, \eqref{Euler12}, \eqref{Euler13} with heat conductivity will be shown in the following arguments. It can be also observed that both constant function and pice-wise constant function satisfy the ideal Euler system \eqref{Euler} with the discontinuity point being arbitrary in the nozzle. However, the heat conductivity $\epsilon-solution$ \textbf{do not} converge to a solution of the inviscid system \eqref{Euler} as $\epsilon$ tends to $0$, even in the weak sense. Specifically speaking, as the uniform state of the flow at the entrance is supersonic, the solutions of system \eqref{Euler11}, \eqref{Euler12}, \eqref{Euler13} are classified as either purely supersonic or transonic (supersonic-subsonic), according to the value of the exit pressure. Moreover, we shall see that a "boundary layer" may be formed for the heat-conductivity solution at either one side of the nozzle as $\epsilon \to 0$. By contrast, as the uniform state of the flow at the entry of the nozzle is subsonic, depending on the exit pressure range, the solutions of system \eqref{Euler11}, \eqref{Euler12}, \eqref{Euler13} can be classified into either purely subsonic solutions or transonic (subsonic-supersonic) solutions. However, the boundary layer forms at both $x=0$ and $x=1$ for the transonic solutions whereas it arises only at the exit of the nozzle for the pure subsonic solutions as $\epsilon$ goes to $0$.  
This paper will give a strict specific description for this layer mathematically. 

It is worth to mention that the "boundary layer" is a singular layer means that the limit of the $\epsilon$-heat conductivity solution $U^{\epsilon}(x)$ at the endpoints can not match with the value of ideal Euler system as the coefficient of heat conductivity $\epsilon$ goes to zero, which is different with the physical boundary layer adjacent to the surface of the body in multi-dimensional space, for instance, see \cite{OleinikSamokhin1999}, \cite{Schlichting2017}.

\subsection{The special solutions for ideal Euler equations \eqref{Euler}}\label{1.1}

As the coefficient of the heat conductivity $\epsilon$ goes to $0$, the equations \eqref{Euler11}, \eqref{Euler12}, \eqref{Euler13} formally converges to the ideal Euler system \eqref{Euler}. Let 
\begin{equation}\label{U0}
    U_{0}:=(q_{0},\rho_{0},p_{0}),
\end{equation}
then for the given uniform supersonic state $U_{in}=U_{0}$ at the entrance, namely $M_{0}>1$, there exists a unique subsonic solution 
\begin{equation}\label{U1}
    U_{1}:= (q_{1}, \rho_{1}, p_{1}),
\end{equation}
which connects the supersonic state $U_{0}$ through a shock, and across the shock front, the following Rankine-Hugoniot conditions hold:
\begin{equation}\label{R-H}
	\begin{aligned}
		 & \rho_{0}q_{0} =\rho_{1}q_{1} ,          \\
		 & \rho_{0}q^{2}_{0}+p_{0}  =\rho_{1}q^{2}_{1}+p_{1}, \\
		 & \Phi_{0}  =\Phi_{1}.
	\end{aligned}
\end{equation}
Then it can be easily verified that this piece-wise smooth function 
\begin{equation}\label{normal_shock_sol}
    \overline{U}(x):=
    \left\{
        \begin{aligned}
            &U_{0},\quad 0\leq x <x_{s},\\
            &U_{1},\quad x_{s} < x \leq 1
        \end{aligned}
    \right.
\end{equation}
gives a normal transonic shock solution to system \eqref{Euler}, where $0<x_{s}<1$ is the position of the shock front and can be arbitrary in the nozzle. Therefore, it seems that there exists infinite transonic shock solution \eqref{normal_shock_sol} to ideal Euler equations \eqref{Euler}.  Moreover, this special normal shock solution satisfies the entropy condition $p_{1}>p_{0}$. It can be also checked that for any $x\in \mathcal{N}$,
\begin{equation}\label{normal_sup}
    \overline{U}_{-}(x):=U_{0} 
\end{equation}
and 
\begin{equation}\label{normal_sub}
  \overline{U}_{+}(x):=U_{1}  
\end{equation}
satisfy system \eqref{Euler}. Namely, functions \eqref{normal_sup} and \eqref{normal_sub} give a constant solution to the inviscid Euler equations \eqref{Euler}. Then we introduce the concepts of special supersonic solution, subsonic solution and  transonic shock solution to the ideal Euler system \eqref{Euler}.
\begin{definition}
   For any $x\in \mathcal{N}$, $\overline{U}_{-}(x)$ ( $\overline{U}_{+}(x)$,  $\overline{U}(x)$) is called the special supersonic solution (correspondingly, subsonic solution and transonic shock solution) to Euler equations \eqref{Euler}.
\end{definition}

\subsection{The singular limit problem and main results }\label{section1.2}

In this subsection, we provide a detailed description of the singular limit problem for fluids with heat conduction. 

Given $U_{in}$ at the entrance and $p_{e}$ at the exit of the nozzle, then we are going to investigate the following singular problem as $\epsilon \to 0+$, and try to give a concrete analysis for the singular layer in detail. That is,

\underline{\textbf{\textit{The singular limit problem (SLP):}}}
\begin{enumerate}
    \item \textbf{For any $\epsilon>0$, try to find a solution for system \eqref{Euler11}, \eqref{Euler12}, \eqref{Euler13} subjected to the following boundary conditions} 
    \begin{align}
        &U(0)=U_{in},\label{bd1}\\
        &p(1)=p_{e};\label{bd2}
    \end{align}
    \item \textbf{As $\epsilon \to 0$, investigate the existence of the solution for the limit equations \eqref{Euler};}
    \item \textbf{As $\epsilon \to 0$, study the asymptotic behavior of the corresponding solution, if exists, for system \eqref{Euler11}, \eqref{Euler12}, \eqref{Euler13} under the boundary conditions \eqref{bd1}, \eqref{bd2}.}
\end{enumerate}

\begin{rem}
   It should be noted that from subsection \ref{1.1}, one can find a special supersonic solution $\overline{U}_{-}(x)$, subsonic solution $\overline{U}_{+}(x)$ and transonic shock solution $\overline{U}(x)$ for the Euler system \eqref{Euler}. Therefore, problem (ii) in the singular limit problem can be solved and we are devoted to deal with (i) and (iii) in this paper. In fact, the existence result for the singular limit problem (i) can be established as the boundary value problem \eqref{Euler11}, \eqref{Euler12}, \eqref{Euler13} equipped with boundary conditions \eqref{bd1} at the entrance and \eqref{bd2} at the exit is well-posedness. However, the structure features of this heat conductivity solutions and the asymptotic behavior as $\epsilon$ goes to $0$ vary depending on the state at the entrance of the nozzle and the value of the pressure taken at the exit.
\end{rem}
Therefore, we present the results corresponding to the supersonic incoming flow and subsonic incoming flow, respectively. Then, the existence theorem for the singular limit problem (i) as the inflow at the entrance is supersonic can be stated as follows:
\begin{thm}\label{thm1}
    Let 
    \begin{equation}
        U_{in}=U_{0}
    \end{equation}
    be a uniform supersonic state defined in \eqref{U0} and suppose the Mach number $\displaystyle M_{0}:=\sqrt{\frac{\rho_{0}q_{0}^{2}}{\gamma p_{0}}}$  at the entrance satisfies one of the following two conditions:
    \begin{enumerate}
		\item[(S1)] $1<\gamma<3$, and 
			\begin{equation}\label{Mach}
			1<M_{0}^{2}\leq \frac{3\gamma-1}{\gamma(3-\gamma)};
			\end{equation}
		\item[(S2)] $\gamma \geq 3$, and $ M_0>1 $.
    \end{enumerate}
    Set 
    \begin{equation}\label{p*A}
         p_{*}\defs\frac{p_{0}+p_{1}}{2},\quad A\defs \rho_{0}q_{0}^{2}+p_{0}.
    \end{equation}
    Then for any $\epsilon>0$ and 
    \begin{enumerate}
        \item $p_{e}\in (0,p_{0})\cup (p_{0},p_{*})$, there exists a supersonic solution $U=U^{\epsilon}(x)\in C^{2}([0,1])$ to the boundary value problem \eqref{Euler11}, \eqref{Euler12}, \eqref{Euler13}, \eqref{bd1}, \eqref{bd2};
        \item $p_{e}\in (p_{*},p_{1})\cup (p_{1},\frac{A}{2})$, the boundary value problem \eqref{Euler11}, \eqref{Euler12}, \eqref{Euler13}, \eqref{bd1}, \eqref{bd2} admits a transonic solution $U=U^{\epsilon}(x)\in C^{2}([0,1])$. 
    \end{enumerate}
    See Table \ref{tab1} below for a brief list for possible cases.
\end{thm}

Some comments on Theorem \ref{thm1} are in order.
\begin{rem}
    The two assumptions (S1) and (S2) in Theorem \ref{thm1} are imposed only as sufficient conditions in the proof to avoid degeneracy of the coefficient in the principle part of the reformulated equations when the exit pressure $p_{e}$ belongs to $(p_{*},p_{1})\cup (p_{1},\frac{A}{2})$. However, when $p_{e}\in (0,p_{0})\cup (p_{0},p_{*})$, the above conclusion remains valid in the absence of assumption (S1) and (S2). More details are given in Section 2. Whether the conclusion of Theorem \ref{thm1} remains valid when both (S1) and (S2) fail is an interesting question that requires further investigation.
\end{rem}
\begin{table}[H]
        \centering
        \begin{tabular}{|c|c|c|c|}
        \hline
            $U(0)=U_{0},p(1)=p_{e}$ &type of solution & \thead{location of \\ boundary layer} &\thead{monotonicity \\of pressure} \\
        \hline
            Case1: $p_{e} \in (0,p_{0})$ & supersonic & $x=1$ &  $\searrow$ \\
        \hline
             Case2: $p_{e} \in (p_{0},p_{*})$ & supersonic & $x=1$ & $\nearrow$\\
        \hline
             Case3: $p_{e} \in (p_{*},p_{1})$ & transonic  & $x=1$ & $\nearrow$ \\
        \hline
             Case4: $p_{e} \in (p_{1},\frac{A}{2})$ & transonic  & $x=0$ & \(\nearrow\)\\
        \hline
        \end{tabular}
        \vspace{3mm}
        \caption{The structure of the solution for \eqref{Euler11}, \eqref{Euler12}, \eqref{Euler13}, \eqref{bd}}
        \label{tab1}
\end{table}

\begin{rem}
    The constant $A$ in Theorem \ref{thm1} is uniquely determined by the uniform inflow state in the nozzle, standing for the quantity of the total momentum. Moreover, the critical pressure $p_{*}$ associating with the sonic state which satisfies the relation
    \begin{equation}
        p_{*}=\frac{A}{\gamma+1}.
    \end{equation}
    Therefore, the solution to problem \eqref{Euler11}, \eqref{Euler12}, \eqref{Euler13} with boundary conditions \eqref{bd} is supersonic when $p_{e}\in (0,p_{0})\cup (p_{0},p_{*})$ and is transonic when $p_{e}\in (p_{*},p_{1})\cup (p_{1},\frac{A}{2})$. Also, it is shown that the pressure is monotonic decreasing when $0<p_{e}<p_{0}$ and is monotonic increasing when $p_{0}<p_{e}<\frac{A}{2}$, see Figure \ref{boundary_layer}. In particular, since the case $p_{e}=p_{0}$ is trival and $p_{e}=p_{1}$ has been thoroughly analyzed by the authors in \cite{FangJiangSun2024JDE}, it will not be further considered in this work.     
\end{rem}

Theorem \ref{thm1} establishes the existence of heat-conductive solutions for each fixed $\epsilon>0$. We next describe the formal asymptotic structure of these solutions as $\epsilon \to 0+$. Since the limiting problem is generally overdetermined by the boundary data \eqref{bd}, the solution $\overline{U}_{-}(x)$, $\overline{U}_{+}(x)$ or $\overline{U}(x)$ may not match the endpoint conditions. The mismatch is therefore lead to a thin layer near one endpoint of the nozzle. To describe such layer more clearly, we introduce the following scaling transformation 
\begin{equation}\label{right}
    y_{1}=\frac{1-x}{\epsilon}
\end{equation}
at $x=1$ for $p_{e}\in (0,p_{0})\cup (p_{0},p_{*})\cup (p_{*},p_{1})$ and
\begin{equation}\label{left}
    y_{2}=\frac{x}{\epsilon}
\end{equation}
at $x=0$ for $p_{e}\in (p_{1},\frac{A}{2})$. Set $ U_{R}(y_{1})\defs U(1-\epsilon y_{1})$ for $p_{e}\in (0,p_{0})\cup (p_{0},p_{*})\cup (p_{*},p_{1})$. Substituting this scaling \eqref{right} into system \eqref{Euler11}, \eqref{Euler12}, \eqref{Euler13} gives the formal equations of the right layer.
\begin{align}
    &\partial_{y_{1}}(\rho_{R}u_{R})=0,\\
    &\partial_{y_{1}}(\rho_{R}u_{R}^{2}+p_{R})=0,\\
    &\partial_{y_{1}}(\rho_{R}u_{R}\Phi_{R})=-\partial_{y_{1}y_{1}}e_{R},
\end{align}
with boundary conditions
\begin{align}
    &p_{R}(0)=p_{e},\\
    &\lim_{y_{1}\to +\infty}U_{R}(y_{1})=U_{0},
\end{align}
where $\displaystyle \Phi_{R}=\frac{1}{2}u_{R}^{2}+\frac{\gamma}{\gamma-1}\frac{p_{R}}{\rho_{R}}$. Similarly, set $U_{L}(y_{2})\defs U(\epsilon y_{2})$ for $p_{e}\in (p_{1},\frac{A}{2})$. Substituting this scaling \eqref{left} into system \eqref{Euler11}, \eqref{Euler12}, \eqref{Euler13}  gives the formal left boundary layer equations 
\begin{align}
    &\partial_{y_{2}}(\rho_{L}u_{L})=0,\\
    &\partial_{y_{2}}(\rho_{L}u_{L}^{2}+p_{L})=0,\\
    &\partial_{y_{2}}(\rho_{L}u_{L}\Phi_{L})=\partial_{y_{2}y_{2}}e_{L},
\end{align}
with boundary conditions
\begin{align}
    &U_{L}(0)=U_{0},\\
    &\lim_{y_{2}\to +\infty}p_{L}(y_{2})=p_{e},
\end{align}
where $\displaystyle \Phi_{L}=\frac{1}{2}u_{L}^{2}+\frac{\gamma}{\gamma-1}\frac{p_{L}}{\rho_{L}}$. Define
\begin{align}
    &U_{R}^{b}(y_{1})=U_{R}(y_{1})-U_{0},\label{boundary_layer1}\\
    &U_{L}^{b}(y_{2})=U_{L}(y_{2})-U_{e}\label{boundary_layer2}.
\end{align}
Here, $U_{R}^{b}(y_{1})$ and $U_{L}^{b}(y_{2})$ are used to describe the profile of the right and left layer. 

Then, under the condition of uniform supersonic flow at the entrance, for the singular limit problem (iii), we have

\begin{thm}\label{thm3}
    Let the assumptions in Theorem \ref{thm1} hold and $U^{\epsilon}(x)\in C^{2}([0,1])$ be a solution to equations \eqref{Euler11}, \eqref{Euler12}, \eqref{Euler13} with boundary conditions \eqref{bd}. Then
    \begin{enumerate}
        \item for $p_{e}\in (0,p_{0})\cup (p_{0},p_{*})\cup (p_{*},p_{1})$, the solution converges to the uniform supersonic state $U_{0}$ away from the exit, and a "boundary layer" forms at $x=1$ of the nozzle. Consequently, the solution has the formal expansion
        \begin{equation}
            U^{\epsilon}(x)=U_{0}+U_{R}^{b}(\frac{1-x}{\epsilon})+r_{R}^{\epsilon}(x),
        \end{equation}
        where $\displaystyle U_{R}^{b}(y_{1})$ is the right boundary layer profile defined as \eqref{boundary_layer1} and $r_{R}^{\epsilon}(x)$ is the remainder term of order $\epsilon$.
        \item for $p_{e}\in (p_{1},\frac{A}{2})$, a "boundary layer" forms at $x=0$ of the nozzle. Consequently, the solution has the formal expansion
        \begin{equation}
            U^{\epsilon}(x)=U_{e}+U_{L}^{b}(\frac{x}{\epsilon})+r_{L}^{\epsilon}(x),
        \end{equation}
        where $\displaystyle U_{L}^{b}(y_{2})$ is the left boundary layer profile defined as \eqref{boundary_layer2} and $r_{L}^{\epsilon}(x)$ is the remainder term of order $\epsilon$.
    \end{enumerate}    
    See Figure \ref{boundary_layer} for a schematic plot of the pressure for different cases.
\end{thm}
\begin{figure}[htbp]
	\centering
	\includegraphics[width=9cm,height=6cm]{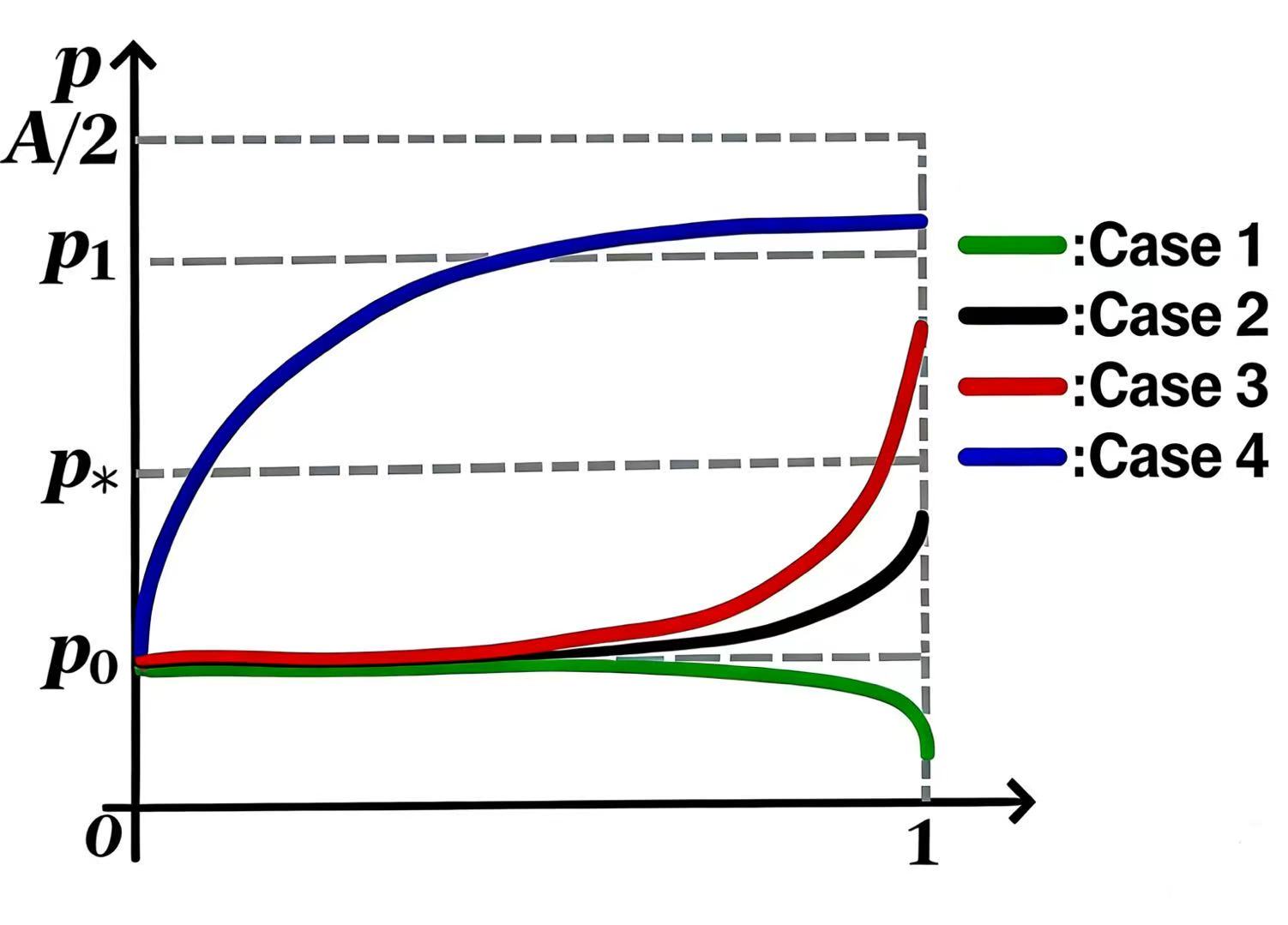}
	\caption{The pressure $p^{\epsilon}(x)$ for $\epsilon>0$ and their limit as $\epsilon \to 0+$.}\label{boundary_layer}
\end{figure}

\begin{rem}
    The remainder terms $r_{R}^{\epsilon}(x)$, $r_{L}^{\epsilon}(x)$ in Theorem \ref{thm3} as well as $r(x)$ in Theorem \ref{thm4} satisfy that for some constant $O(1)$,
    \begin{equation}
        r_{R}^{\epsilon}(x)=O(1)\epsilon,\hspace{1.5mm}r_{L}^{\epsilon}(x)=O(1)\epsilon,\hspace{1.5mm}r(x)=O(1)\epsilon.
    \end{equation}
    Explicit formulas for $r_{R}^{\epsilon}(x)$, $r_{L}^{\epsilon}(x)$ and $r(x)$ require further detailed computations.
\end{rem}
\begin{rem}
    Theorem \ref{thm3} reveals that the solution for Euler system with heat conduction can not converge (even in the weak sense) to either the  supersonic solution or the transonic shock solution of the ideal Euler system. This phenomenon is due to the incompatibility between the exit boundary conditions and the Euler solution, rendering the boundary layer nonvanishing in the limit. Moreover, the location of the boundary layer is governed by the state at the entrance and the exit pressure $p_{e}$, see Table \ref{tab1} and Figure \ref{boundary_layer}.
\end{rem}
\begin{rem}
    The boundary layer in Theorem \ref{thm3} differs fundamentally from the interior shock front, which connects the supersonic state and subsonic state by the Rankine-Hugoniot conditions, see \cite{FangJiangSun2024JDE}. Instead, this layer studied herein is a boundary-induced singular layer rather than an interior discontinuity spontaneously generated within the flow. This two phenomena reflect two distinct mechanisms in the same vanishing heat-conductivity limit: the interior shock front selects an admissible interior shock under compatible shock data, whereas the boundary-induced singular layers indicate that a singular layer forms at the endpoint which leads to the failure of the convergence with general boundary conditions.
\end{rem}
So far, the preceding results focus on the case where the prescribed incoming state is the uniform supersonic state $U_{0}$. It is then natural to ask whether the solution exists and what is the asymptotic behavior as the coefficient of the heat conductivity goes to $0$ when the boundary data at the entrance of the nozzle are instead by the corresponding uniform subsonic state $U_{1}$. In fact, the change in the incoming flow leads to a different boundary-layer structure, as stated in the following theorems. We first give the existence result for the singular limit problem (i) as the incoming flow is subsonic:
\begin{thm}\label{thm2}
    Given
    \begin{equation}
        U_{in}=U_{1}
    \end{equation}
    at the entrance,
	where $U_{1}$ is a uniform subsonic state defined in \eqref{U1}. Then under assumptions $(S1)$ or $(S2)$ in Theorem \ref{thm1}, 
    for any $\epsilon>0$, there exists a solution $U=U^{\epsilon}(x)\in C^{2}([0,1])$ to the system \eqref{Euler11}, \eqref{Euler12}, \eqref{Euler13} with boundary conditions \eqref{bd1}, \eqref{bd2}. More precisely,
    \begin{itemize}
        \item if $p_{e}\in (0,p_{*})$, then the solution $U^{\epsilon}(x)$ is transonic (subsonic-supersonic);
        \item if $p_{e}\in (p_{*},p_{1})\cup (p_{1},\frac{A}{2})$, then the solution $U^{\epsilon}(x)$ is subsonic.
    \end{itemize}
    See Table \ref{tab2} and Figure \ref{singular_layer} for more information.
\end{thm}  
\begin{table}[H]
		\centering
		\begin{tabular}{|c|c|c|c|}
			\hline
			$U(0)=U_{1},p(1)=p_{e}$ &type of solution & \thead{location of \\ boundary layer} &\thead{monotonicity \\of pressure} \\
			\hline
			Case5: $p_{e} \in (0,p_{*})$ & sub-sup & $x=0,1$ &  $\searrow$\\
			\hline
			Case6: $p_{e} \in (p_{*},p_{1})$ & subsonic & $x=1$ &  $\searrow$ \\
			\hline
			Case7: $p_{e} \in (p_{1},\frac{A}{2})$ & subsonic & $x=1$ &  $\nearrow$\\
			\hline
		\end{tabular}
		\vspace{3mm}
		\caption{The structure of the solution for \eqref{Euler11}, \eqref{Euler12}, \eqref{Euler13},\eqref{bd}}
		\label{tab2}
	\end{table}

\begin{rem}
    Unlike Theorem \ref{thm1}, which imposes a uniform supersonic inflow at the entrance of the nozzle, Theorem \ref{thm2} is devoted to the case where the boundary value at $x=0$ is prescribed with a uniform subsonic incoming flow. Moreover, it can be observed that the pressure is monotonic decreasing when $p_{e}\in (0,p_{*})\cup (p_{*},p_{1})$ and is monotonic increasing when $p_{e}\in (p_{1},\frac{A}{2})$, see Table \ref{tab2}. The assumptions (S1) and (S2) are still required for technical purposes in the proof of the existence for the solution, whether the conclusions remain valid when neither (S1) nor (S2) holds requires further verification.
\end{rem}

The next result describes the singular behavior of these solutions as the heat conductivity coefficient vanishes for the singular limit problem (iii). 
\begin{thm}\label{thm4}
    Let $U^{\epsilon}(x)$ be the solution obtained in Theorem \ref{thm2}. Then as $\epsilon \to 0+$, the following boundary-layer structures occur.
    \begin{itemize}
        \item for $p_{e}\in (0,p_{*})$, the solution $U^{\epsilon}(x)$ converges away from both endpoints to $U_{*}\defs (u_{*},\rho_{*},p_{*})$, where $p_{*}$ is defined as \eqref{p*A}. Moreover, boundary layers form simultaneously near $x=0$ and $x=1$, so the solution has the formal expansion
        \begin{equation}
            U^{\epsilon}(x)=U_{*}+U_{L}^{b}(\frac{x}{\epsilon})+U_{R}^{b}(\frac{1-x}{\epsilon})+r^{\epsilon}(x),
        \end{equation}
        where $U_{R}^{b}(x)$, $U_{L}^{b}(x)$ are defined as \eqref{boundary_layer1}, \eqref{boundary_layer2} and $r^{\epsilon}(x)$ is the remainder term of order $\epsilon$.        
        \item for $p_{e}\in (p_{*},p_{1})\cup (p_{1},\frac{A}{2})$, the solution converges away from the exit to the uniform subsonic state $U_{1}$. A boundary layer forms near $x=1$. Formally,
        \begin{equation}
            U^{\epsilon}(x)=U_{1}+U_{R}^{b}(\frac{1-x}{\epsilon})+r_{R}^{\epsilon}(x),
        \end{equation}
        where $U_{R}^{b}(x)$ is defined as \eqref{boundary_layer1} and $r_{R}^{\epsilon}(x)$ is the remainder term of order $\epsilon$.
    \end{itemize}
\end{thm}
\begin{figure}[htbp]
	\centering
	\includegraphics[width=8.5cm,height=5.5cm]{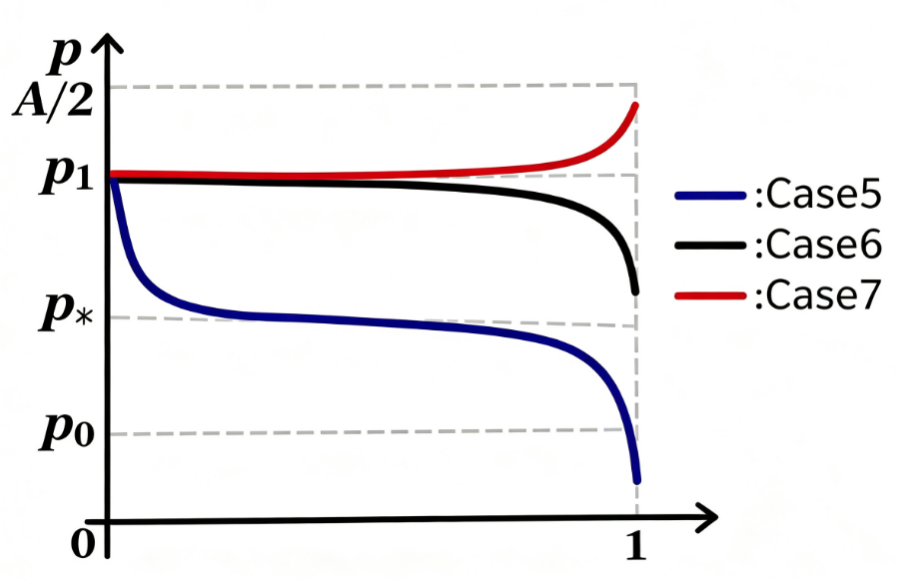}
	\caption{The pressure $p^{\epsilon}(x)$ for $\epsilon>0$ and their limit as $\epsilon \to 0+$.}\label{singular_layer}
\end{figure}
\begin{rem}
 Analogously to Theorem \ref{thm3}, the heat conduction solutions $U^{\epsilon}(x)$ fail to converge (even in the weak sense) to the corresponding subsonic solution or transonic shock solutions of inviscid Euler equations as $\epsilon \to 0$. However, dual boundary layers simultaneously form at both endpoints $x=0$ and $x=1$ for the subsonic-supersonic solution when $p_{e}\in (0,p_{*})$ and the boundary layer localizes solely at $x=1$ regardless of whether $p_{e}$ lies above or below the threshold $p_{1}$ for $p_{e}\in (p_{*},\frac{A}{2})$, which is different with the cases in Figure \ref{singular_layer}. The reasons for this phenomenon remain to be further investigated.
\end{rem}

Next, we briefly summarize the existing research on pure supersonic solution, pure subsonic solution and transonic shock solution for inviscid flows and flows with physical parameters, as well as the vanishing viscosity limit to the inviscid compressible flow.
For the ideal fluid, where the friction, heat-transfer, heat conductivity and other physical effects are neglected, the study of supersonic solution, subsonic solution and transonic shock solutions in a nozzle has yielded extensive results, see \cite{EmbidGoodmanMajda1984SIAM}, \cite{FangXin2021CPAM}, \cite{Liu1982CMP}, \cite{CourantFriedrichs1948},  \cite{Liu1982ARMA}, \cite{WangXin2015SIAM}, \cite{WangXin2019SIAM}, \cite{XinYin2005CPAM} and references therein. However, the physical parameters that are neglected in the flows may play an important role when considering the existence of the corresponding solutions. Therefore, we pay attention to the non-ideal Euler system and investigate the existence of solution and asymptotic behavior of the solutions when certain parameters tend to zero. Indeed, for a scalar conservation law, the existence, uniqueness, and global stability of vanishing viscosity solutions were established by Oleinik \cite{oleinik1963AMS} in one space dimension.  See also \cite{Bianchini2005AM, FangZhao2021CPAA, Gilbarg1951AJM, GoodmanXin1992ARMA, Hoff1989India, Liu1985, Yu1999ARMA} and the references therein on the inviscid limit to the ideal Euler equations.
By considering physical parameters, heat-conductivity, with shock boundary data at the entrance and at the exit of the nozzle in \cite{FangJiangSun2024JDE}, it is verified that for the polytropic gases the heat conductivity shock solution exists and converges to a normal shock solution for the Euler equations. It is worth to mention that we extend the result here in \cite{FangJiangSun2024JDE} to a general boundary data when considering heat-conductivity. See also \cite{gues2004,gues2002} for related studies in multi-dimensional spaces. 

This paper is organized as follows. Section \ref{section1} presents the mathematical formulation of the singular limit problem for fluids with heat conductivity and the main results. Specifically, Theorem \ref{thm1} and Theorem \ref{thm2} are devoted to the existence of heat-conductive transonic shock solutions under general boundary conditions, while Theorem \ref{thm3} and Theorem \ref{thm4} characterize the asymptotic behavior of these shock solutions as the heat-conductivity parameter tends to zero.
Section \ref{section2} is dedicated to the detailed proofs of the main results. In Subsection \ref{section2.1}, the original governing equations \eqref{Euler11}-\eqref{Euler13} are reformulated into a more tractable equivalent form. Then the detailed proofs of Theorem \ref{thm1} and Theorem \ref{thm3} for the case of uniform supersonic incoming flow are provided in Subsection \ref{section2.2}, followed by the proofs of Theorem \ref{thm2} and Theorem \ref{thm4} for the uniform subsonic incoming flow in Subsection \ref{section2.3}.

\section{Reformulation for problem (SLP) and proof of the main theorems}\label{section2}
\subsection{Reformulation for the equations \eqref{Euler11}, \eqref{Euler12}, \eqref{Euler13}}\label{section2.1}
The fluids with heat conduction investigated in this paper can be stated as:
\begin{align}
	&\partial_{x}(\rho u) =0,\label{heat_Euler1} \\
	& \partial_{x}(\rho u^{2}+p) =0, \label{heat_Euler2}\\
	& \partial_{x}(\rho u \Phi) =\epsilon e_{xx},\label{heat_Euler3}
\end{align}
where $\epsilon$ is the coefficient of heat conductivity.

By employing equation $\eqref{heat_Euler1}$ and $\eqref{heat_Euler2}$, one can assume
\begin{align}
    &\rho u\equiv \rho_{0}q_{0}=1,\label{equiv1}\\
    &\rho u^{2}+p=\rho_{0}q_{0}^{2}+p_{0}\equiv A
\end{align}
Therefore, the Bernoulli function $\Phi$ can be written as
\begin{equation}\label{re_Phi}
	\Phi=\frac{1}{2}u^{2}+\frac{\gamma}{\gamma-1}\frac{p}{\rho}=\frac{1}{2}(A-p)^{2}+\frac{\gamma}{\gamma-1}p(A-p)\defs\Phi(p).
\end{equation}

Set $f(p)=\Phi(p)-\Phi(p_{0})$, obviously, $f(p)$ is a smooth function with respect to $p$ defined in $(0,+\infty)$ and it is strictly concave, namely $f''(p)<0$, see Figure \ref{Bernoulli}. 
\begin{figure}[htbp]
	\centering	\includegraphics[width=9.5cm,height=5.3cm]{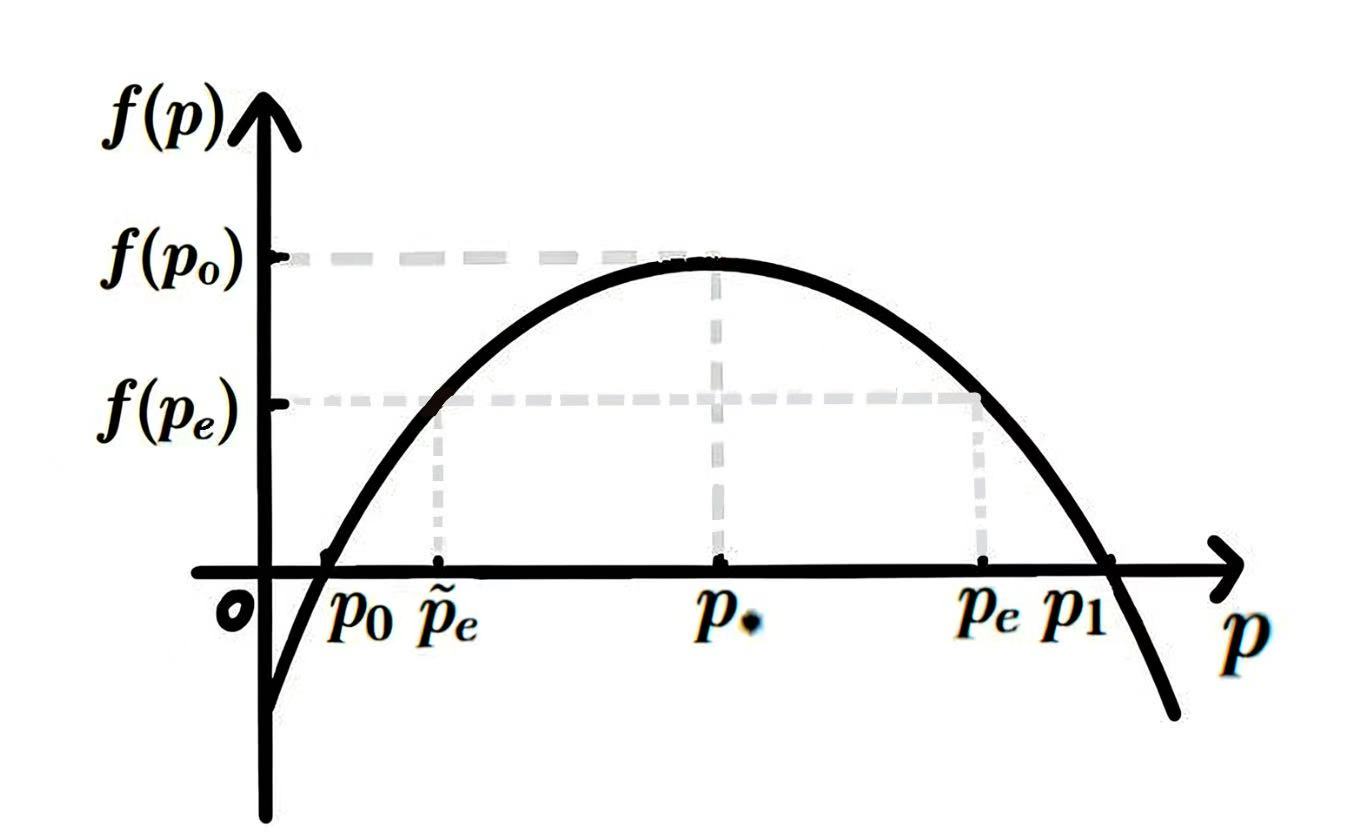}
	\caption{The auxiliary function of $f(p)$.}\label{Bernoulli}
\end{figure}

By applying the Rankine-Hugoniot condition $\eqref{R-H}_{3}$, it can be checked
\begin{equation}
    f(p_{0})=f(p_{1})=0,
\end{equation}
which implies for any $p\in (p_{0},p_{1})$, one has $f(p)>0$. Moreover, it can be checked that $f'(p_{*})=0$, where $p_{*}$ is defined as \eqref{p*A}. 
Denote $\displaystyle \kappa=\frac{\epsilon}{\gamma-1}$, by employing the identity equality \eqref{equiv1}, \eqref{p*A} as well as the formula \eqref{re_Phi}, system \eqref{heat_Euler1}, \eqref{heat_Euler2}, \eqref{heat_Euler3} can be reduced into a single equation
\begin{equation}\label{eq:2rd}
\kappa (A-2p^{\kappa})\partial_{x}p^{\kappa}=F(p^{\kappa},\alpha_{\kappa})
\end{equation}
where
\begin{equation}
    F(p,\alpha):=f(p)+\alpha,
\end{equation}
and $\alpha_{\kappa}:=\kappa (A-2p_{in})\partial_{x}p^{\kappa}(0)$ is an unknown constant to be determined together with $ p^{\kappa} $ under the boundary conditions 
\begin{align}
    &p^{\kappa}(0)=p_{in},\label{bd3}\\
    &p^{\kappa}(1)=p_{e}.\label{bd4}
\end{align}
Then solving system \eqref{heat_Euler1}, \eqref{heat_Euler2}, \eqref{heat_Euler3} with boundary conditions \eqref{bd1}, \eqref{bd2} is reduced to solving boundary value problem \eqref{eq:2rd}, \eqref{bd3}, \eqref{bd4}.

\subsection{Proof of Theorem \ref{thm1} and Theorem \ref{thm3}}\label{section2.2}
In this subsection, we will give the details of the proof for Theorem \ref{thm1}. Namely, we are going to deal with equations \eqref{heat_Euler1}, \eqref{heat_Euler2}, \eqref{heat_Euler3} subjected to the following boundary conditions
\begin{align}
	&U(0)=U_{0},\label{sup_bd1}\\
	&p(1)=p_{e},\label{sup_bd2}
\end{align}
where $p_{e}\in (0,\frac{A}{2})$ is a constant satisfying $p_{e}\neq p_{1}$ and $p_{e}\neq p_{0}$.
Let $p^{\kappa}$ be a solution to system \eqref{eq:2rd}, \eqref{bd3}, \eqref{bd4}, one has
\begin{lem}\label{lem3-1}
    Given $U_{0}$ and $p_{e}$ defined as \eqref{sup_bd1}, \eqref{sup_bd2}.
	Then if $p^{\kappa}\in C^{2}(0,1)\cap C^{1}([0,1])$ is a solution to the boundary value problem \eqref{eq:2rd} and \eqref{bd3}-\eqref{bd4}, there holds
    \begin{equation}\label{prior}
        \min \{p_{0},p_{e}\}<p^{\kappa}(x)<\max \{p_{0},p_{e}\}, \quad x\in (0,1).
    \end{equation}
    Moreover, for $0<p_{e}<p_{0}$, it holds
	\begin{equation}\label{derivative1}
		\partial_{x}p^{\kappa}(0)<0,\quad \partial_{x}p^{\kappa}(1)<0;
	\end{equation}
	for $p_{0}<p_{e}<\frac{A}{2}$, it holds
	\begin{equation}\label{derivative2}
		\partial_{x}p^{\kappa}(0)>0,\quad \partial_{x}p^{\kappa}(1)>0.
	\end{equation}
\end{lem}
\begin{proof}
By applying the assumptions (S1) or (S2) and lemma A.2 in \cite{FangJiangSun2024JDE} to equation \eqref{eq:2rd}, it can be deduced that $p^{\kappa}(x)$ is a non-decreasing function, which implies  \eqref{prior}. Moreover, equations \eqref{heat_Euler1}, \eqref{heat_Euler2}, \eqref{heat_Euler3} can be reduced as
\begin{equation}\label{3-1}
	\kappa (A-2p)\partial_{xx}p-2\kappa(\partial_{x}p)^2-\partial_{x}(\frac{1}{2}(A-p)^{2}+\frac{\gamma}{\gamma-1}p(A-p))=0
\end{equation}
using \eqref{equiv1} and \eqref{p*A}.
Then the assumptions (S1) or (S2) implies that  \eqref{3-1} is a second order elliptic partial differential equation. Therefore, \eqref{derivative1}, \eqref{derivative2} can be obtained by applying the Hopf lemma.
\end{proof}

Once Lemma \ref{lem3-1} holds, equation \eqref{eq:2rd} implies the strict monotonicity of $p^{\kappa}(x)$. Namely, it is strictly decreasing for $0<p_{e}< p_{0}$, and strictly increasing for $p_{0} < p_{e} < \frac{A}{2}$. Correspondingly, it follows from \eqref{derivative1} and \eqref{derivative2} that $\alpha_{\kappa} < 0$ in the former case and $\alpha_{\kappa} > 0$ in the latter. Then there exists an inverse function $X_{\kappa}(p)$ of $p^{\kappa}(x)$, satisfying
\begin{align}
    & \frac{dX_{\kappa}(p)}{dp}=  \frac{\kappa(A-2p)}{F(p,\alpha_{\kappa})}, \label{eq:inverse}\\
	& X_{\kappa}(p_{0})=0, \quad  X_{\kappa}(p_{e})=1.\label{bd:inverse} 
\end{align}
The solution of this ODE boundary value problem \eqref{eq:inverse},\eqref{bd:inverse} can be described as
\begin{equation}
    X_{\kappa}(p)=\int_{p_{0}}^{p}\frac{\kappa (A-2s)}{F(s;\alpha_{\kappa})} ds,
\end{equation}
where $\alpha_{\kappa}$ satisfies 
\begin{equation}
    1=\int_{p_{0}}^{p_{e}}\frac{\kappa (A-2p)}{F(p;\alpha_{\kappa})} dp.
\end{equation}
Moreover, 
Next, it suffices to prove that 
\begin{enumerate}
    \item the existence and convergence of $\alpha_{\kappa}$;
    \item the convergence of $p^{\kappa}(x)$,
    \item the limit $p^{0}$ and $\alpha_{0}$ satisfy 
    \begin{equation}\label{weak-sol1-hp}
        F(p^{0};\alpha_{0})=0
    \end{equation}
    in the weak sense. Namely, for any $\phi(x)\in C^{\infty}([0,1])$, it holds
    \begin{equation}\label{non-weak-sol1-hp}
        \int_{0}^{1}F(u^{0};\alpha_{0})\phi(x) dx=0.
    \end{equation}
    \item the convergence of $X_{\kappa}(p)$.
\end{enumerate}

For (i), we have
\begin{lem}\label{lem:exists_alpha}
    For any $\kappa>0$, there are exists a unique $\alpha_{\kappa}$ such that
    \begin{equation}\label{qqq}
        Q_{\kappa}(\alpha_{\kappa})=1,
    \end{equation}
    where $\displaystyle Q_{\kappa}(\alpha)=\int_{p_{0}}^{p_{e}}\frac{\kappa (A-2p)}{F(p;\alpha)}dp$.
\end{lem}
\begin{proof}
    We divided the proof into three cases: 
    \begin{itemize}
        \item For $p_{e}\in(0,p_{0})$:  By Lemma \ref{lem3-1}, it can be verified $\alpha_{\kappa}<0$ and $Q_{\kappa}(\alpha)$ is a continuous, strictly increasing function with respect to $\alpha<0$. Then 
        \begin{align}
            Q_{\kappa}(\alpha)&=\int_{p_{0}}^{p_{e}}\frac{\kappa (A-2p)}{F(p;\alpha)}dp\notag\\
            &\leq \int_{p_{0}}^{p_{e}}\frac{\kappa A}{f^{\prime}(p_{0})(p-p_{0})+\alpha}dp\notag\\
            &=\frac{\kappa A}{f^{\prime}(p_{0})}ln\frac{f^{\prime}(p_{0})(p_{e}-p_{0})+\alpha}{\alpha}\notag\\
            &\to +\infty ,\quad \text{as} \quad \alpha \to 0-. \label{ww}
        \end{align}
        Moreover, $\displaystyle \lim_{\alpha \to -\infty}Q_{\kappa}(\alpha)=0$ and \eqref{ww} indicates that there exists a unique $\alpha_{\kappa}$ such that \eqref{qqq} holds.
        \item For $p_{e}\in(p_{0},p_{*})\cup  (p_{*},p_{1})$: Similarly, Lemma \ref{lem3-1} implies $\alpha_{\kappa}>0$. As $Q_{\kappa}(\alpha)$ is a strictly continuous decreasing function with respect to $\alpha\in (0,+\infty)$, then
    \begin{equation}\label{bb}
        \lim_{\alpha \to +\infty}Q_{\kappa}(\alpha)=0.
    \end{equation}
    On the other hand,
    \begin{align}
			&\int_{p_{0}}^{p_{e}}\frac{\kappa(A-2p)}{f(p)+\alpha}dp\notag \\
			& >\int_{p_{0}}^{p_{e}}\frac{\kappa A}{f'(p_{0})(p-p_{0})+\alpha}dp-\int_{p_{0}}^{p_{e}}\frac{2\kappa p}{f'(p_{0})(p-p_{0})+\alpha}dp  \notag \\
			& =\frac{\kappa A}{f'(p_{0})}\ln\left(1+\frac{f'(p_{0})(p_{e}-p_{0})}{\alpha}\right)  \notag \\
			& \quad - \left [ \frac{2\kappa(p_{e}-p_{0})}{f'(p_{0})}+\frac{2\kappa}{f'(p_{0})}\left(p_{0}-\frac{\alpha}{f'(p_{0})}\right)\ln\left(1+\frac{f'(p_{0})(p_{e}-p_{0})}{\alpha}\right)\right]\notag \\
			& =-\frac{2\kappa(p_{e}-p_{0})}{f'(p_{0})}+\frac{\kappa}{f'(p_{0})}\left(A-2p_{0}+\frac{2\alpha}{f'(p_{0})}\right)\ln\left(1+\frac{f'(p_{0})(p_{1}-p_{0})}{\alpha}\right)\notag\\
            &\to +\infty \quad\text{as}\quad \alpha \to 0+.\label{aa}
		\end{align}
        Therefore, \eqref{bb} and \eqref{aa} imply that there exists a unique $\alpha_{\kappa}$ such that $Q_{\kappa}(\alpha_{\kappa})=1$.
        \item  For $p_{e}\in (p_{1},\frac{A}{2})$:  One can deduce from Lemma \ref{lem3-1} that $\alpha_{\kappa}>-f(\frac{A}{2})>0$ and $Q_{\kappa}(\alpha)(\alpha)$ is a strictly continuous decreasing function with respect to $\alpha$. Moreover,
        \begin{align*}
            Q_{\kappa}(\alpha)&=\int_{p_{0}}^{p_{e}}\frac{\kappa (A-2s)}{f(s)+\alpha}ds\\
            &\leq\int_{p_{0}}^{p_{e}}\frac{\kappa(A-2s)}{f^{\prime}(p_{e})(s-p_{e})+f(p_{e})+\alpha}ds\\
            &\leq\int_{p_{0}}^{p_{e}}\frac{\kappa(A-2p_{1})}{f^{\prime}(p_{e})(s-p_{e})+f(p_{e})+\alpha}ds\\
            &=\frac{A-2p_{1}}{f^{\prime}(p_{e})}\ln \frac{f(p_{e})+\alpha}{f^{\prime}(p_{e})(p_{0}-p_{e})}\\
            &\to +\infty, \quad \text{as} \hspace{1mm}\alpha \to -f(p_{e}).
        \end{align*}
        Therefore, 
        \begin{equation}
            \lim_{\alpha\to -f(p_{e})}Q_{\kappa}(\alpha)=+\infty. 
        \end{equation}
        Then we finish the proof for this lemma.
    \end{itemize}
\end{proof}
As an immediate consequence of Lemma \ref{lem:exists_alpha}, we have:
\begin{lem}
    Let $\kappa>0$. Then
    \begin{itemize}
        \item for given $p_{e}\in (0,p_{*})$,  there exists a supersonic solution $(p^{\kappa},\alpha_{\kappa})$ satisfying the equation \eqref{eq:2rd}  and the boundary conditions \eqref{bd3},\eqref{bd4}.
        \item for given $p_{e}\in (p_{*},\frac{A}{2})$, there exists a transonic shock solution $(p^{\kappa},\alpha_{\kappa})$ satisfying the equation \eqref{eq:2rd}  and the boundary conditions \eqref{bd3},\eqref{bd4}.
    \end{itemize}
\end{lem}

Now we have proved the existence of the solution to the inverse boundary value problem \eqref{eq:inverse}, \eqref{bd:inverse}, which shows that for every $\kappa>0$, there exists a $p^{\kappa}(x)$ satisfying \eqref{eq:2rd}, \eqref{bd3}, \eqref{bd4}. In the following arguments, we are going to investigate the asymptotic behavior of the solution $p^{\kappa}$ as $\kappa$ vanishes.

In fact, as $p^{\kappa}$ is strictly monotonic, then the estimates \eqref{prior} implies 
\begin{equation}
    T.V.p^{\kappa}=|p_{e}-p_{0}|.
\end{equation}
By employing Helly's theorem, there exists a subsequence $p^{\kappa_{n}}$ and a bounded variation function $p^{0}$ such that for every $n\geq 1$ and $a.e. x\in [0,1]$,
\begin{equation}
    p^{\kappa_{n}}\to p^{0}, \quad \text{as} \quad n\to \infty.
\end{equation}
Thus, we have proved the convergence of $p^{\kappa}(x)$. Next, it suffice to check that the uniqueness of $p^{0}$ and to prove $p^{0}$ is a weak solution to equation \eqref{weak-sol1-hp}. Firstly, we have to investigate the limit behavior of $\alpha^{\kappa}$. 
\begin{lem}\label{lem:2-4}
    For any $\kappa>0$, 
    \begin{itemize}
        \item when $p_{e}\in (0,p_{0})\cup(p_{0},p_{*})\cup  (p_{*},p_{1})$, it holds
    \begin{equation}
        \lim_{\kappa \to 0}\alpha_{\kappa}=0;
    \end{equation}
        \item when $p_{e}\in (p_{1},\frac{A}{2})$, it holds
        \begin{equation}\label{lim_alpha21}
            \lim_{\kappa \to 0}\alpha_{\kappa}=-f(p_{e}).
        \end{equation}
    \end{itemize} 
\end{lem}
\begin{proof}
    \textbf{(a). For $0<p_{e}<p_{0}$,} define
	\begin{equation}
		\tilde{L}(p)=\frac{f(p_{e})}{p_{e}-p_{0}}(p-p_{0})\defs l(p-p_{0}),
	\end{equation}
    where $\displaystyle l=\frac{f(p_{e})}{p_{e}-p_{0}}>0$.
    Obviously, for any $p\in (p_{e},p_{0})$, it holds
    \begin{equation}
        0>f(p)>\tilde{L}(p).
    \end{equation}
    Then we have
    \begin{align}
        Q_{\kappa}(\alpha) & = \int_{p_{0}}^{p_{e}}\frac{\kappa(A-2p)}{f(p)+\alpha}dp  \notag \\
        &\leq \int_{p_{0}}^{p_{e}} \frac{\kappa A}{f(p)+\alpha}dp\notag \\
        &\leq \int_{p_{0}}^{p_{e}} \frac{\kappa A}{\tilde{L}(p)+\alpha}dp\notag \\
        &= \int_{p_{0}}^{p_{e}} \frac{\kappa A}{l(p-p_{0})+\alpha}dp\notag \\
        &=\frac{\kappa A}{l}ln\frac{l(p_{e}-p_{0})+\alpha}{\alpha}\notag \\
        &\defs \tilde{Q}_{\kappa}(\alpha).
    \end{align}
    By solving $\tilde{Q}_{\kappa}(\alpha)=1$, we have $\displaystyle \tilde{\alpha}_{\kappa}=\frac{l(p_{e}-p_{0})}{e^{\frac{l}{\kappa A}}-1}$. Since $Q_{\kappa}(\alpha)$ is strictly increasing for any $\alpha\in (-\infty,0)$ and 
    \begin{equation}
        \lim_{\kappa \to 0}\tilde{\alpha_{\kappa}}=0.
    \end{equation}
    Then $\tilde{\alpha}_{\kappa}<\alpha_{\kappa}<0$ implies $\lim_{\kappa \to 0}\alpha_{\kappa}=0$.
    
	\textbf{(b). For $p_{0}<p_{e}<p_{*}$,} define:
	\[
		\tilde{L}(p)=
		\begin{cases}s(p-p_{0}),\quad   
        &p_{0}\leq p\leq p_{*},  \\
             -s(p-p_{1}), \quad 
        &p_{*}\leq p\leq p_{1},
		\end{cases}
	\]
	where
	\begin{equation}
		s:=\frac{f(p_{*})}{p_{*}-p_{0}}=-\frac{f(p_{*})}{p_{*}-p_{1}}>0
	\end{equation}
	Obviously, $0\leq \tilde{L}(p)\leq f(p)$ for any $p\in (p_{0},p_{1})$. Thus
	\begin{align}
		Q_{\kappa}(\alpha) & = \int_{p_{0}}^{p_{e}}\frac{\kappa(A-2p)}{f(p)+\alpha}dp  \notag \\
        &\leq \int_{p_{0}}^{p_{1}}\frac{\kappa(A-2p)}{f(p)+\alpha}dp  \notag \\
		& \leq \int_{p_{0}}^{p_{1}}\frac{\kappa A}{\hat{L}(p)+\alpha}dp \notag  \\
       &=\left(\int_{p_{0}}^{p_{*}}+\int_{p_{*}}^{p_{1}}\right)\frac{\kappa A}{\hat{L}(p)+\alpha}dp \notag\\
	   & =\frac{2\kappa A}{s}\ln(1+\frac{s}{2\alpha}(p_{1}-p_{0})) \notag \\
	   & =:\tilde{Q}_{\kappa}(\alpha)
	\end{align}
	Let $\tilde{Q}_{\kappa}(\alpha)=1$, then we have
	\begin{equation}
		\alpha=\frac{s(p_{1}-p_{0})}{2(e^{\frac{s}{2\kappa A}}-1)}=:\tilde{\alpha}_{\kappa}
	\end{equation}
	Obviously, $$\lim_{\kappa\to 0}\tilde{\alpha}_{\kappa}=0.$$ Then from $0<\alpha_{\kappa}<\hat{\alpha}_{\kappa}$, we have $\lim_{\kappa \to 0}\alpha_{\kappa}=0$.

    \textbf{(c). For $p_{1}<p_{e}<\frac{A}{2}$,} $Q_{\kappa}(\alpha)$ is a decreasing function with respect to $\alpha$. Let 
    \begin{equation}
        l:=\frac{f(p_{e})}{p_{e}-p_{1}}<0.
    \end{equation}
    Then for any $\kappa>0$,
    \begin{align}
        Q_{\kappa}(\alpha)&=\int_{p_{0}}^{p_{e}}\frac{\kappa (A-2s)}{f(s)+\alpha}ds\\
        &=(\int_{p_{0}}^{p_{1}}+\int_{p_{1}}^{p_{e}})\frac{\kappa (A-2s)}{f(s)+\alpha}ds\\
        &\leq \kappa \int_{p_{0}}^{p_{1}}\frac{A}{f(p)}dp+\int_{p_{1}}^{p_{e}}\frac{\kappa A}{l(p-p_{1})+\alpha}dp\\
        &\leq \kappa \int_{p_{0}}^{p_{1}}\frac{A}{f(p)}dp+\frac{\kappa A}{l}ln \frac{l(p_{e}-p_{1})+\alpha}{\alpha}\\
        &=\kappa C_{p_{0},p_{1}}+\frac{\kappa A}{l}\ln \frac{f(p_{e})+\alpha}{\alpha}\\
        &:=\tilde{Q}(\alpha).
    \end{align}
    where $C_{p_{0},p_{1}}$ is a constant depending on $p_{0}$, $p_{1}$.
    Set 
    \begin{equation}\label{lim}
        \tilde{Q}_{\kappa}(\tilde{\alpha})=1,
    \end{equation}
    it holds
    \begin{equation}
        \tilde{\alpha}_{\kappa}=\frac{f(p_{e})}{e^{(1-\kappa C_{p_{0},p_{1}})\frac{l}{\kappa A}}-1}.
    \end{equation}
    Therefore, $\displaystyle -f(p_{e})< \alpha_{\kappa}<\tilde{\alpha}_{\kappa}$.
    Let $\kappa$ goes to $0$, we have $\lim_{\kappa \to 0}\alpha_{\kappa}=-f(p_{e})$.   
\end{proof}
\begin{rem}
    From the expression of 
    \begin{equation}
        \alpha_{\kappa}=\kappa (A-2p_{0}) \partial_{x}p(0)
    \end{equation}
    and \eqref{lim_alpha2} in Lemma \ref{lem:2-4}, we have
    \begin{equation}
        \lim_{\kappa \to 0}\partial_{x}p(0)=+\infty,
    \end{equation}
    which indicates that the pressure changes sharply at the the entrance and there exists a "boundary layer" at  $x=0$. Elaborate calculations can be seen in the following sequence.
\end{rem}

Now, it can be verified that for any $\phi(x)\in C^{\infty}([0,1])$, it holds
\begin{align}
    &\quad \kappa_{n}\int_{0}^{1}(A-2p^{\kappa_{n}})\partial_{x}p^{\kappa_{n}}\phi(x)dx\notag \\
	&=\kappa_{n} A p^{\kappa_{n}}\phi(x)|_{x=0}^{x=1}  -\kappa_{n}\int_{0}^{1}(A-2p^{\kappa_{n}})p^{\kappa_{n}}\partial_{x}\phi(x) dx\notag \\
	&\quad -\kappa_{n}(p^{\kappa_{n}})^{2}\phi(x)|_{x=0}^{x=1}  +\kappa_{n}\int_{0}^{1}(p^{\kappa_{n}})^{2}\partial_{x}\phi(x)dx\notag\\
    &\quad \to 0,\quad \text{as} \quad \kappa_{n} \to 0.
\end{align}
Moreover, 
\begin{equation*}
		\lim_{\kappa_{n}\to 0}\int_{0}^{1}F(p^{\kappa_{n}},\alpha_{\kappa_{n}})\cdot \phi (x)dx=\int_{0}^{1}F(p^{0},\alpha_{0})\cdot \phi(x)dx.
\end{equation*}
Therefore, the limit function $p^{0}$ is a weak solution to equation \eqref{weak-sol1-hp}. Next, we are going to investigate the limit behavior of the function $X_{\kappa}$. Let 
\begin{equation}\label{I}
     I_{\kappa}(p):=\frac{X_{\kappa}(p)}{1-X_{\kappa}(p)}:=\frac{\int_{p_{0}}^{p}\frac{A-2s}{F(s;\alpha_{\kappa})}ds}{\int_{p}^{p_{e}}\frac{A-2s}{F(s;\alpha_{\kappa})}ds}
\end{equation}
Then it can be verified that 
\begin{lem}\label{lem:last}
 Suppose $I_{\kappa}(p)$ is defined as \eqref{I}, then    
 \begin{itemize}
     \item[(1)] when $p_{e}\in (0,p_{0})$, it holds that
     \begin{equation}\label{lim_I1}
         \lim_{\kappa \to 0+}I_{\kappa}(p)=+\infty;
     \end{equation}
     \item[(2)] when $p_{e}\in (p_{0},p_{*})\cup (p_{*},p_{1})$, it holds that
     \begin{equation}\label{lim_I2}
         \lim_{\kappa \to 0+}I_{\kappa}(p)=+\infty;
     \end{equation}
     \item[(3)] when $\displaystyle p_{e}\in(p_{1},\frac{A}{2})$, it holds that
      \begin{equation}\label{lim_I}
        \lim_{\kappa \to 0+}I_{\kappa}(p)=0.
    \end{equation}
 \end{itemize}  
\end{lem}
\begin{proof}
Set
\begin{equation*}
    I_{\kappa}^{1}(p)=\int_{p_{0}}^{p}\frac{A-2s}{F(s;\alpha_{\kappa})}ds,\quad I_{\kappa}^{2}(p)=\int_{p}^{p_{e}}\frac{A-2s}{F(s;\alpha_{\kappa})}ds.
\end{equation*}

\textit{\textbf{Case 1: $0<p_{e}<p_{0}$.}} 
It can be calculated that
\begin{align}
    &\int_{p_{0}}^{p}\frac{A-2s}{F(s;\alpha_{\kappa})}ds=\int_{p_{0}}^{p}\frac{A-2s}{f(s)+\alpha_{\kappa}}ds>\int_{p_{0}}^{p}\frac{A-2s}{f^{\prime}(p_{0})(p-p_{0})+\alpha_{\kappa}}ds\notag\\
    &>\int_{p_{0}}^{p}\frac{A-2s}{f^{\prime}(p_{0})(p-p_{0})+\alpha_{\kappa}}ds=\frac{A}{f^{\prime}(p_{0})}\ln \left|\frac{f^{\prime}(p_{0})(p-p_{0})+\alpha_{\kappa}}{\alpha_{\kappa}}\right|\notag\\
    &\to +\infty,\quad \text{as} \quad\kappa \to 0,
\end{align}
which implies that
\begin{equation}\label{lem2.6_1}
    \lim_{\kappa \to 0}\int_{p_{0}}^{p}\frac{A-2s}{F(s;\alpha_{\kappa})}ds \to \infty.
\end{equation}
Moreover, for any $p\in (0,p_{0})$, choose $\delta>0$, such that $\delta>f(p)$, then it holds
\begin{equation}
    f(0)+\alpha_{\kappa}\leq f(s)+\alpha_{\kappa}\leq \delta + \alpha_{\kappa}<0,\quad \text{for}\hspace{1mm} \text{any}\hspace{1mm} s\in(p_{e},p)
\end{equation}
as well as
\begin{equation}
    A-2p_{0}\leq A-2s< A.
\end{equation}
Thus 
\begin{equation}
   0< \frac{A}{-\delta+\alpha_{\kappa}}(p_{e}-p)<\int_{p}^{p_{e}}\frac{A-2s}{F(s;\alpha_{\kappa})}ds\leq \frac{A-2p_{0}}{f(0)+\alpha_{\kappa}}(p_{e}-p),
\end{equation}
which implies
\begin{equation}\label{lem2.6_2}
    \frac{A}{-\delta}(p_{e}-p)<\lim_{\kappa \to 0}I_{\kappa}^{2}(p)<\frac{A-2p_{0}}{f(0)}(p_{e}-p).
\end{equation}
since $\lim_{\kappa \to 0}\alpha_{\kappa}=0$. Combine \eqref{lem2.6_1}, \eqref{lem2.6_2}, one obtains \eqref{lim_I1}.

Then for any $p\in(0,p_{0})$, since $\alpha_{\kappa}<0$, it follows that 
\begin{align}
	&\lim_{\kappa \to 0+}I^{1}_{\kappa}(p)=+\infty;\label{Q1}\\
    &\lim_{\kappa \to 0+}I^{2}_{\kappa}(p)=\frac{A-2p_{*}}{2\nu (p_{1}-p_{0})}ln \frac{(p_{e}-p_{1})(p-p_{0})}{(p_{e}-p_{0})(p-p_{1})}-\frac{1}{\nu}ln\frac{(p_{e}-p_{1})(p_{e}-p_{0})}{(p-p_{1})(p-p_{0})},\label{Q2}
\end{align}
which implies \eqref{lim_I1}. Combine \eqref{Q1} and \eqref{Q2}, then \eqref{lim_I2} holds.

\textit{\textbf{Case 2,3: $p_{0}<p_{e}<p_{1}$ and $p_{e}\neq p_{*}$.}} 

Rewrite $F(p;\alpha)$ as
\begin{align*}
    F(p;\alpha)&=f(p)+\alpha\\
    &=-\frac{\gamma+1}{2(\gamma-1)}p^{2}+\frac{A}{\gamma-1}p+\frac{1}{2}A^{2}-\Phi(p_{0})+\alpha\\
    &=\nu ((p-p_{*})^{2}-\mu ^{2}),
\end{align*}
where $\alpha>0$,
\begin{equation}
    \nu =-\frac{\gamma+1}{2(\gamma-1)}<0,\quad \mu^{2}=\frac{\nu (p_{0}-p_{1})^{2}-4\alpha}{4 \nu}>0.
\end{equation}
Therefore,
\begin{align}
    &I_{\kappa}^{1}(p)=\frac{1}{\nu}\int_{p_{0}}^{p}\frac{A-2s}{(s-p_{*})^{2}-\mu^{2}}ds\notag \\
    &=\frac{A-2p_{*}}{2\nu \mu}\int_{p_{0}}^{p}\frac{1}{s-p_{*}-\mu}-\frac{1}{s-p_{*}+\mu}ds-\frac{1}{\nu}\int_{p_{0}}^{p}\frac{2(s-p_{*})}{(s-p_{*})^{2}-\mu^{2}}ds\notag \\
    &=\frac{A-2p_{*}}{2\nu \mu}ln \frac{(p-p_{*}-\mu)(p_{0}-p_{*}+\mu)}{(p-p_{*}+\mu)(p_{0}-p_{*}-\mu)}-\frac{1}{\nu}ln\frac{(p-p_{*}-\mu)(p-p_{*}+\mu)}{(p_{0}-p_{*}-\mu)(p_{0}-p_{*}+\mu)}
\end{align}
and 
\begin{equation}
    I_{\kappa}^{2}(p)=\frac{A-2p_{*}}{2\nu \mu}ln \frac{(p_{e}-p_{*}-\mu)(p-p_{*}+\mu)}{(p_{e}-p_{*}+\mu)(p-p_{*}-\mu)}-\frac{1}{\nu}ln\frac{(p_{e}-p_{*}-\mu)(p_{e}-p_{*}+\mu)}{(p-p_{*}-\mu)(p-p_{*}+\mu)}.
\end{equation}
Let $\alpha$ goes to $0$, we have
\begin{equation}
    \mu\to \pm\frac{p_{0}-p_{1}}{2}.
\end{equation}
Since $\alpha_{\kappa}>0$ for $p_{e} \in (p_{0},\frac{A}{2})$. Then we have
\begin{align}
    &\lim_{\kappa \to 0+}I_{\kappa}^{1}(p)= +\infty,\\
    &\lim_{\kappa \to 0+}I_{\kappa}^{2}(p)=C_{p},
\end{align}
where $C_{p}$ is a constant depending on $p$. Therefore, we have \eqref{lim_I2}.

\textit{\textbf{Case 4: $p_{1}<p_{e}<\frac{A}{2}$.}} 

Let $\kappa$ tends to 0, by applying Lemma \ref{lem:2-4}, we have that
\begin{equation}
    \alpha_{\kappa} \to -f(p_{e}),
\end{equation}
which indicates that
\begin{equation}
    \mu \to \pm (p_{e}-p_{*}),
\end{equation}
where $f(p_{e})=\nu ((p_{e}-p_{*})^{2}-\frac{(p_{0}-p_{1})^{2}}{4})$. 
Apparently, 
\begin{equation}
    p_{e}-p_{*} \neq \frac{p_{1}-p_{0}}{2}
\end{equation}
which implies there is no singularity on $[p_{0},\frac{A}{2})$ for integral $I_{\kappa}^{1}(p)$. Namely, for any $p\in (p_{0},p_{e})$, 
\begin{equation}
    \lim_{\kappa \to 0}I_{\kappa}^{1}(p)=C_{p}
\end{equation}
where $C_{p}$ is a constant depending on $p$. However, for $I_{\kappa}^{2}(p)$, it can be checked that
\begin{equation}
    \lim_{\kappa \to 0}I_{\kappa}^{2}(p)=+\infty.
\end{equation}
Thus $\lim_{\kappa \to 0}I_{\kappa}(p)=0$.

\end{proof}
\begin{rem}
	In the proof of Lemma \ref{lem:last}, Case 2 and Case 3 are combined because $\displaystyle \int_{z_{1}}^{z_{2}}\frac{A-2s}{F(s;\alpha_{\kappa})}ds$ is a singular integral either $z_{1}$ or $z_{2}$ equal to $p_{0}$, $p_{1}$ as $\alpha$ goes to zero. The difference between them lies in the nature of the solution: for $0<p_{e}<p_{*}$, the solution is supersonic and whereas $p_{*}<p_{e}<p_{1}$, it is transonic.
\end{rem}
Set
\begin{equation}\label{XS}
    X_{s}\defs
	\begin{cases}
		1, & \text{if } p_{e}\in(0,p_{1})/\{p_{0}\} \\
        0, & \text{if } p_{e}\in(p_{1},\frac{A}{2})
	\end{cases}
\end{equation}
It follows immediately that 
\begin{equation}\label{lim_xs}
    \lim_{\kappa \to 0}X_{\kappa}(p)=X_{s},
\end{equation}
which yields that
\begin{equation*}
	\lim_{\kappa\to 0+}\int_{p_{0}}^{p_{e}}|X_{\kappa}(p)-X_{s}|dp=0
\end{equation*}
This means that $\{X_{\kappa}(p)\}$ is a Cauchy sequence in $L^{1}(p_{0},p_{e})$, which implies that $p^{\kappa}(x)$ is also a Cauchy sequence in $L^{1}(0,1)$. Hence, there exists a $p_{*}^{0}(x)\in L^{1}(0,1)$ 
such that
\begin{equation}
\lim_{\kappa\to 0+}\int_{0}^{1}|p^{\kappa}(x)-p_{*}^{0}(x)|dx=0.
\end{equation}
By employing \eqref{lim_xs}, it can be checked that for $p_{e}\in (0,p_{1})/\{p_{0}\}$, it holds
\begin{equation}\label{qq}
    p_{*}^{0}(x)=
	\begin{cases}
		&p_{0},\quad 0\leq x < X_{s}\\
		&p_{1},\quad  X_{s}<x\leq 1
	\end{cases}
\end{equation}
where $X_{s}$ is defined as \eqref{XS}. Then \eqref{lim_xs} and \eqref{qq} indicate that there is a singular layer at $x=1$ as the value of the boundary condition at the endpoint $p_{e}\neq p_{1}$. Moreover, for $p_{e}\in (p_{1},\frac{A}{2})$, it can be verified from \eqref{lim_alpha21} in Lemma \ref{lem:2-4} that the solution $p^{\kappa}(x)$ undergoes an immediate jump, taking the value $p_{e}$ at $x=1$ from its boundary condition $p_{0}$ at $x=0$. 
Then, for any pointwise convergent subsequence $p^{\kappa_{n}},(n\geq 1)$, with the limit function $p^{0}$, we have
\begin{equation}
    p^{0}(x)=p^{0}_{*}(x),\quad \text{a.e.} \quad x\in (0,1).
\end{equation}
Therefore, \eqref{qq} indicates that a singular layer may form for the $\epsilon$-heat conduction solutions at the left endpoint $x=0$ and right endpoint $x=1$ when $p_{e}$ belongs to $\displaystyle (p_{1},\frac{A}{2})$ and $(0,p_{1})/\{p_{0}\}$ respectively. Then Theorem \ref{thm1} follows once Lemma \ref{lem:last} holds.

\subsection{Proof of Theorem \ref{thm2} and Theorem \ref{thm4}}\label{section2.3}
In this subsection, we aim to study the existence of the subsonic solution and transonic shock solution for system \eqref{heat_Euler1}, \eqref{heat_Euler2}, \eqref{heat_Euler3} with the following boundary conditions
\begin{align}
	&U(0)=U_{1};\label{bd1_sub}\\
	&p(1)=p_{e},\label{bd2_sub}
\end{align}
where $p_{e}$ is a constant that satisfies $0<p_{e}<\frac{A}{2}$ and $p_{e}\neq p_{0}, p_{1}$. Next we are going to deal with \textbf{Case 5-7} in Theorem \ref{thm2} respectively. 

Set $\displaystyle \kappa=\frac{\epsilon}{\gamma-1}$, similar arguments as in Lemma \ref{lem3-1}, we have
\begin{equation}\label{sub_prior}
    \min \{p_{1},p_{e}\}<p^{\kappa}(x)<\max \{p_{1},p_{e}\},\quad x\in (0,1).  
\end{equation}
Moreover, for $p_{e}\in (0,p_{1})$, it holds
\begin{equation}\label{sub_hopf1}
    \partial_{x}p^{\kappa}(0)<0,\quad \partial_{x}p^{\kappa}(1)<0;
\end{equation}
for $p_{e}\in (p_{1},\frac{A}{2})$, we have
\begin{equation}\label{sub_hopf2}
    \partial_{x}p^{\kappa}(0)>0,\quad \partial_{x}p^{\kappa}(1)>0.
\end{equation}

Set $F(p;\alpha):=f(p)+\alpha$ with $f(p)=\Phi(p)-\Phi(p_{0})$. Then the boundary value problem \eqref{heat_Euler1}, \eqref{heat_Euler2}, \eqref{heat_Euler3} together with boundary conditions \eqref{bd1_sub}, \eqref{bd2_sub} can be reduced as
\begin{align}
    &\kappa (A-2p^{\kappa})\partial_{x}p^{\kappa}=F(p^{\kappa},\alpha_{\kappa}),\label{sub_eq}\\
    &p^{\kappa}(0)=p_{1},\label{sub_bd1}\\
    &p^{\kappa}(1)=p_{e}.\label{sub_bd2}
\end{align}

Let $\kappa>0.$ Our aim in this subsection is trying to determine $(p^{\kappa},\alpha_{\kappa})$ satisfying equations \eqref{sub_eq} subjected to the boundary conditions \eqref{sub_bd1}, \eqref{sub_bd2}. As $\kappa \to 0+$, equation \eqref{sub_eq} formally tends to 
\begin{equation}
    F(p^{0},\alpha_{0})=0.
\end{equation}
Then try to determine $(u^{0},\alpha_{0})$ in the weak sense: for any test function $\phi(x)\in C^{\infty}([0,1])$, there holds
\begin{equation}
    \int_{0}^{1}F(u^{0};\alpha_{0})\phi(x)dx=0,
\end{equation}
and satisfies the boundary conditions \eqref{sub_bd1}, \eqref{sub_bd2}.

\subsubsection{Unique existence of $(u^{\kappa},\alpha_{\kappa})$ and the formation of the boundary layer}

It can be verified from \eqref{sub_hopf1}, \eqref{sub_hopf2} and the equation \eqref{sub_eq}  that $p^{\kappa}(x)$ is strictly decreasing for $p_{e}\in (0,p_{1})$ and strictly increasing for $p_{e}\in (p_{1},\frac{A}{2})$. Therefore, the inverse function $X_{\kappa}(p)$ of pressure $p^{\kappa}(x)$ exists and satisfies
\begin{align}
    &\frac{d X_{\kappa}(p)}{dp}=\frac{\kappa(A-2p)}{F(p,\alpha_{\kappa})},\label{ode_sub1}\\
    &X_{\kappa}(p_{1})=0,\quad X_{\kappa}(p_{e})=1.\label{ode_sub2}
\end{align}
where $\alpha_{\kappa}=\kappa \partial_{x}p(0)$. Then solving the reduced boundary value problem \eqref{sub_eq}, \eqref{sub_bd1}, \eqref{sub_bd2} is equivalent to deal with this ordinary differential equations \eqref{ode_sub1}, \eqref{ode_sub2}. The solution can be described as
\begin{equation}\label{sol_ode}
    X_{\kappa}(p)=\int_{p_{1}}^{p}\frac{\kappa (A-2s)}{F(s;\alpha_{\kappa})}ds,
\end{equation}
where $\alpha_{\kappa}$ satisfies
\begin{equation}
    1=\int_{p_{1}}^{p_{e}}\frac{\kappa (A-2s)}{F(s;\alpha_{\kappa})}ds.
\end{equation}

Next, we are going to prove the existence of $\alpha_{\kappa}$. One can follow the arguments in Subsection 2.2, with slightly modifications, to conclude the following lemma.
\begin{lem}
    For any $\kappa>0$, 
    \begin{itemize}
        \item for $p_{e}\in (0,p_{1})$, it holds
        \begin{equation}\label{lim_alpha1}
            \lim_{\kappa \to 0}\alpha_{\kappa}=-f(p_{*});
        \end{equation}
        \item for $p_{e}\in (p_{*},p_{1})\cup (p_{1},\frac{A}{2})$, it holds
        \begin{equation}\label{lim_alpha2}
            \lim_{\kappa \to 0}\alpha_{\kappa}=-f(p_{e}).
        \end{equation}
    \end{itemize}
\end{lem}
Finally, in order to study the asymptotic behavior of the solution $p^{\kappa}(x)$ as  $\kappa$ goes to $0$, one has to investigate the limit behavior of $X_{\kappa}(p)$ when $\kappa \to 0$. The following lemma follows. 
\begin{lem}\label{lem29}    
    Define
    \begin{equation}\label{X-I}
        I_{\kappa}(p):=\frac{X_{\kappa}(p)}{1-X_{\kappa}(p)}=\frac{\int_{p_{1}}^{p}\frac{A-2s}{F(s;\alpha_{\kappa})}ds}{\int_{p}^{p_{e}}\frac{A-2s}{F(s;\alpha_{\kappa})}ds}. 
    \end{equation}
    Then
    \begin{itemize}
        \item[(1)] for $p_{e}\in (0,p_{1})$, it holds
        \begin{equation}\label{lim_I0}
            \lim_{\kappa \to 0}I_{\kappa}(p)=
        \begin{cases}
            +\infty,\quad &0<p<p_{*}\\
            0,\quad &p_{*}<p<p_{1};
        \end{cases} 
        \end{equation}
        \item[(2)] for $p_{e}\in  (p_{*},p_{1})\cup(p_{1},\frac{A}{2})$, it holds
        \begin{equation}\label{lim_I11}
            \lim_{\kappa \to 0}I_{\kappa}(p)=+\infty.
        \end{equation}
    \end{itemize}
\end{lem}
\begin{proof}
    Rewrite $F(s;\alpha)$ as
    \begin{equation}
        F(s;\alpha)=\nu(s-p_{*})^{2}+g(p_{0})+\alpha<0,
    \end{equation}
    where $\displaystyle \nu=-\frac{\gamma +1}{2(\gamma-1)}$, $\displaystyle g(p_{0})=\frac{\gamma^{2}}{2(\gamma^{2}-1)}A^{2}-\Phi(p_{0})$.

\textbf{ Proof of (1):} By applying \eqref{sub_hopf1}, we have $F(p,\alpha_{\kappa})<0$, which implies
 \begin{equation}
     g(p_{0})+\alpha_{\kappa}<0.
 \end{equation}
 Moreover, one can deduce from \eqref{lim_alpha1} that
 \begin{equation}
     \lim_{\kappa \to 0}\alpha_{\kappa}=-g(p_{0}).
 \end{equation}
 Therefore,
    \begin{align}
        \int_{p_{1}}^{p}\frac{A-2s}{F(s;\alpha_{\kappa})}ds&=\int_{p_{1}}^{p}\frac{A-2s}{\nu(s-p_{*})^{2}+g(p_{0})+\alpha_{\kappa}}ds\notag\\
        &=\int_{p_{1}}^{p}\frac{A}{\nu(s-p_{*})^{2}+g(p_{0})+\alpha_{\kappa}}ds-\int_{p_{1}}^{p}\frac{2s}{\nu(s-p_{*})^{2}+g(p_{0})+\alpha_{\kappa}}ds\notag\\
        &=-\frac{A-2p_{*}}{\sqrt{\nu(g(p_{0})+\alpha_{\kappa})}}arctan \sqrt{\frac{\nu}{g(p_{0}+\alpha_{\kappa})}}(s-p_{*})|_{s=p_{1}}^{s=p}\notag \\
        &\quad - \frac{1}{\nu} \ln \{\nu(s-p_{*})^{2}+g(p_{0})+\alpha\}|_{s=p_{1}}^{s=p}\notag \\
        &=\frac{A-2p_{*}}{\sqrt{\nu(g(p_{0})+\alpha_{\kappa})}}\{\arctan \sqrt{\frac{\nu}{g(p_{0})+\alpha_{\kappa}}}(p_{1}-p_{*})\notag\\
        &\quad-\arctan \sqrt{\frac{\nu}{g(p_{0})+\alpha_{\kappa}}}(p-p_{*})\}\notag\\
        &\quad - \frac{1}{\nu} \ln \frac{\nu(p-p_{*})^{2}+g(p_{0})+\alpha_{\kappa}}{\nu(p_{1}-p_{*})^{2}+g(p_{0})+\alpha_{\kappa}} \notag\\
        &:=J_{\kappa}^{1}+J_{\kappa}^{2},\label{q1}
    \end{align}
    where 
\begin{align}
    J_{\kappa}^{1}&=\frac{A-2p_{*}}{\sqrt{\nu(g(p_{0})+\alpha_{\kappa})}}\{\arctan \sqrt{\frac{\nu}{g(p_{0})+\alpha_{\kappa}}}(p_{1}-p_{*})\notag\\
    &\quad-\arctan \sqrt{\frac{\nu}{g(p_{0})+\alpha_{\kappa}}}(p-p_{*})\},\\
    J_{\kappa}^{2}&=- \frac{1}{\nu} \ln \frac{\nu(p-p_{*})^{2}+g(p_{0})+\alpha_{\kappa}}{\nu(p_{1}-p_{*})^{2}+g(p_{0})+\alpha_{\kappa}}.
\end{align}       
    Then for $p_{*}<p<p_{1}$, it can be verified that
    \begin{align}
        &\lim_{\kappa \to 0} J_{\kappa}^{1}=-\frac{(A-2p_{*})(p_{1}-p)}{\nu(p-p_{*})(p_{1}-p_{*})}>0,\\
        &\lim_{\kappa \to 0}J_{\kappa}^{2}=-\frac{2}{\nu}\ln \frac{p-p_{*}}{p_{1}-p_{*}}>0,
    \end{align}
    which implies
    \begin{equation}\label{lem29_1}
        \lim_{\kappa \to 0}\int_{p_{1}}^{p}\frac{A-2s}{F(s;\alpha_{\kappa})}ds=-\frac{(A-2p_{*})(p_{1}-p)}{\nu(p-p_{*})(p_{1}-p_{*})}-\frac{2}{\nu}\ln \frac{p-p_{*}}{p_{1}-p_{*}}<\infty.
    \end{equation}
    Similarly, one can checked that
    \begin{align}
        \lim_{\kappa \to 0}\int_{p}^{p_{*}}\frac{A-2s}{F(s;\alpha_{\kappa})}ds&=\lim_{\kappa \to 0}\frac{A-2p_{*}}{\sqrt{\nu(g(p_{0})+\alpha_{\kappa})}}\arctan \sqrt{\frac{\nu}{g(p_{0})+\alpha_{\kappa}}}(p-p_{*})\notag\\
        &\quad -\lim_{\kappa \to 0}\frac{1}{\nu}\ln \frac{g(p_{0})+\alpha_{\kappa}}{\nu(p-p_{*})^{2}+g(p_{0})+\alpha_{\kappa}}\notag\\
        &\to +\infty,\label{lem29_2}\\
        \lim_{\kappa \to 0}\int_{p_{*}}^{p_{e}}\frac{A-2s}{F(s;\alpha_{\kappa})}ds&=\lim_{\kappa \to 0}-\frac{A-2s}{\sqrt{\nu(g(p_{0})+\alpha_{\kappa})}}\arctan \sqrt{\frac{\nu}{g(p_{0})+\alpha_{\kappa}}}(p_{e}-p_{*})\notag\\
        &-\frac{1}{\nu}\lim_{\kappa \to 0} \ln \frac{\nu(p_{e}-p_{*})^{2}+g(p_{0})+\alpha_{\kappa}}{g(p_{0})+\alpha_{\kappa}}\notag\\
        & \to +\infty,\label{lem29_3} 
    \end{align}
    which implies that for any $0<p_{e}<p_{*}$,
    \begin{equation}
        \lim_{\kappa \to 0} \int_{p}^{p_{e}}\frac{A-2s}{F(s;\alpha_{\kappa})}ds \to +\infty.
    \end{equation}
    Analogously, for $0<p<p_{*}$, one obtains that
    \begin{align}
        &\lim_{\kappa \to 0}\int_{p_{1}}^{p}\frac{A-2s}{F(s;\alpha_{\kappa})}ds\to +\infty,\label{lem29_4}\\
        &\lim_{\kappa \to 0}\int_{p}^{p_{e}}\frac{A-2s}{F(s;\alpha_{\kappa})}ds \to C_{0}<\infty,\label{lem29_5}
    \end{align}
    where $C_{0}$ is a constant depending on $\nu, p_{0}, p_{1}$. Therefore, combine \eqref{lem29_1}-\eqref{lem29_5}, one has \eqref{lim_I0}.
    
\textbf{ Proof of (2):} When $p_{e}\in (p_{*},p_{1})$, since $\displaystyle \lim_{\kappa\to 0}\alpha_{\kappa}=-f(p_{e})$, it can be easily verified that $\displaystyle \int_{p_{1}}^{p}\frac{A-2s}{F(s;\alpha_{\kappa})}$ is a finite constant depending on $p_{e}$ and $\displaystyle \int_{p}^{p_{e}}\frac{A-2s}{F(s;\alpha_{\kappa})}ds$ goes to $+\infty$ as $\kappa \to 0$. As for $p_{e}\in (p_{1},\frac{A}{2})$, one can follow the proof for Case 4 in Lemma \ref{lem:last} with slight modification to get \eqref{lim_I11}. Therefore, we have finished the proof for this lemma.
\end{proof}

From \eqref{X-I}, we have 
\begin{equation}
        \lim_{\kappa \to 0} X_{\kappa}(p)=
        \begin{cases}
            1,\quad &0<p< p_{*}\\
            0,\quad &p_{*}<p<p_{1}.\\
        \end{cases} 
\end{equation}
and
\begin{equation}
    \lim_{\kappa \to 0}X_{\kappa}(p)=1,
\end{equation}
for $p_{e}\in (0,p_{1})$ and $p_{e}\in (p_{1},\frac{A}{2})$ respectively.

Therefore, for the boundary value problem \eqref{sub_eq}-\eqref{sub_bd2}, one has the following conclusion:
\begin{lem}\label{lem210}
    For equations \eqref{sub_eq} with boundary conditions \eqref{sub_bd1}, \eqref{sub_bd2},
    \begin{itemize}
        \item given $p_{e}\in (0,p_{*})$, there exists a subsonic-supersonic solution $U=U^{\epsilon}(x)\in C^{2}([0,1])$. Moreover, as $\epsilon$ goes to $0$, a boundary layer appears at $x=0$ and $x=1$ of the nozzle;
        \item given $p_{e}\in (p_{*},\frac{A}{2})/\{p_{1}\}$, there exists a subsonic solution $U=U^{\epsilon}(x)\in C^{2}([0,1])$. Moreover, as $\epsilon$ goes to $0$, a boundary layer appears at the endpoint $x=1$ of the nozzle.
    \end{itemize}
\end{lem}
\begin{proof}
    This lemma is an immediate corollary of \eqref{sol_ode} and Lemma \ref{lem29}.
\end{proof}
Once Lemma \ref{lem210} holds, we have finished the proof of Theorem \ref{thm2}.

\section*{Acknowlegements} 
The research of the paper was supported in part by the Scientific Research Funds of Xiamen University of Technology YKJ25061R, and Fujian Natural Science Foundation of China (NO.2026J008314).


\bibliographystyle{abbrv}
\bibliography{ref1}

\end{document}